\documentclass[a4paper,leqno,14pt]{article}

\usepackage{epsfig}
\usepackage{graphicx}
\usepackage{amsmath,amsthm,amssymb,latexsym}
\usepackage{amsfonts}
\usepackage[american]{babel}

\newcommand{\Tan}{{\mathrm{Tan}}\,}
\newcommand{\Nor}{{\mathrm{Nor}}\,}

\newcommand{\vare}{\varepsilon}

\newcommand{\R}{{\Bbb R}}

\newcommand{\E}{{\Bbb E}}
\newcommand{\N}{{\Bbb N}}

\def\j1n{j=1,\dots,n}
\def\j1m{j=1,\dots,m}
\def\i1np1{\in +1}

\def\R{\mathbb{R}}
\def\N{\mathbb{N}}

\def\i1np1{\in +1}

\def\R{\mathbb{R}}

\def\u1{u^{(1)}}

\def\h1{h^{(1)}}

\newcommand{\la}{\lambda}

\newcommand{\RR}{{\mathbb R}}

\newcommand{\pa}{\partial}

\newcommand{\dist} {\mbox{\rm dist }}

\usepackage{subfigure}

\theoremstyle{plain}
\newtheorem{theorem}{Theorem}[section]
\newtheorem{corollary}[theorem]{Corollary}
\newtheorem{lemma}[theorem]{Lemma}
\newtheorem{proposition}[theorem]{Proposition}

\newtheorem{remark}[theorem]{Remark}

\theoremstyle{definition}
\newtheorem{definition}{Definition}
\newtheorem{example}{Example}

\begin{document}

\begin{title}{On the complements of union of open balls and related set classes.}
\end{title}

\author{M. Longinetti \footnote{Marco.Longinetti@unifi.it, Dipartimento DIMAI,
Università degli Studi di Firenze, V.le Morgagni 67/a,  50134
Firenze - Italy}\\  P. Manselli \footnote{Università degli Studi di Firenze.}  \\ A.Venturi \footnote{Università degli Studi di Firenze.}}

\maketitle

\begin{abstract}  Let a $R$-body be a closed set, complement of union of open balls of radius $R$ in the Euclidean space. Properties generalizing similar ones for convex sets are proved for the family of $R$-bodies; properties for the family of sets supported by spheres of
radius $R$ (extension of the supporting hyperplane to convex bodies) are investigated. Comparison of that family with the sets of reach $R$ and with the $R$-rolling sets \cite{Wal} are studied. New properties for the previous  families are proved, by using the $R$-cones, generalization of the convex cones. 
\end{abstract}

{\bf AMS Subjet Classification:} {Primary: 52A01;  Secondary: 52A30}

{\bf Keywords: }{ generalized convexity,  convex hull,  rolling bodies, reach positive sets, support cones.}

\section{Introduction}
Looking for classes of sets more general than the convex sets, one has to look: i) for properties of the convex sets maintained and ii) for properties of the convex sets generalized.

 Here a family of sets is studied: the sets that are the complements  of a non empty union of open balls of  radius $R$, generalizing the property of the closed convex sets as complements of union of open half spaces.  As the family of convex sets, this family of sets is closed with respect to the intersection (see  Proposition \ref{precedentiproprieta} and Remark \ref{minimalfamily}).
 
  This family has been introduced by Perkal \cite{Perkal},  used in Walther \cite{Wal}, in Cuevas, Fraiman, Pateiro-L\'opez \cite{Cue} and in \cite{MLVRbodies}; the family of these sets is  called $R$-bodies  there. 
 
    Here  further properties of the $R$-bodies are provided, also in comparison with other classes of sets.  
    
    In \S \ref{Rbodiessection} the $R$-bodies  are introduced;  some of their properties, proved in \cite{Perkal} and in \cite{MLVRbodies}, are recalled. Other  formulas for the $R$-hulloid of a body $E$, the minimal $R$-body which contains $E$, denoted by $co_R(E)$, are made  explicit in Theorem \ref{inviluppoR-conigeneral} and Corollary \ref{formulainviluppoR-conigeneral}.
    
     In \S  \ref{Rsupportsection} it has been introduced the following definition:
  an open  ball $B$ of radius  $R$ is  $R$-supporting the body $A$ at $a$, when $a\in \pa B$ and $B$ is in the complement of $A$.  
The set $\mathcal{N}_R(A,a)$ of the unit vectors $v$,  normal to a $R$ supporting ball B at $a\in \pa A $ is introduced, generalizing the property of supporting half space for convex sets.   The properties of the intersection of $A$ with the boundary of a $R$-supporting ball at a  boundary point of $co_R(A)$, are studied in Theorem \ref{theorem1suportingball}.

Let us call $A$ a $R$-supported body if $A$ is a body and  $\mathcal{N}_R(A,a)$ is non empty for all $a\in \pa A$.
The  family of $R$-supported bodies contains the family of $R$-bodies.
In \cite{Cue} these sets were called $R$-rolling sets;  in   \cite[Proposition 2]{Cue} the authors introduced a body  with a $R$-supporting unit vectors at every boundary point which is not a $R$-body, see  Remark \ref{viceversaEggleston} for more examples.

In \S \ref{R-bodies,reach,R-rolling sets}
  the $R$-bodies and the $R$-supported bodies have been matched with the sets of reach greater or equal than $R$ and with bodies $E$ having  balls rolling freely inside $E$ and in the complement of $E$, \cite{Wal}. 
  
In  \S  \ref{section R-cones} the family  of $R$-cones is introduced.
An $R$-cone with vertex $x$ is the $R$-body obtained by intersection of the complements of a given family of open  balls of radius $R$, having $x$ on their boundaries. If 
$ R \to \infty$ the $R$-cones have limit convex cones.  
In  Corollary \ref{inviluppoR-coni}    a representation of a $R$-body $A$ by  $R$-  cones is given.

In \S \ref{subsectioncones} the relations between  $\mathcal{N}_R(A,a)$,  the  tangent set $\Tan(A,a) $ and the normal cone $\Nor(A,a)$  are investigated, see Theorems \ref{viceversaregular} and \ref{viceversaregular2}.

A characterization of the family of  sets of reach greater or equal than $R$
is obtained in Theorems \ref{regulararereach}, \ref{regulararereachd=2}   through the property of convexity of $\mathcal{N}_R(A,a)$.

\section{Definitions and preliminaries}\label{preliminaires}

 As in  \cite{Schn}, a non empty closed subset of the Euclidean space $ \RR^d$,   will be called a  {\em body}.

Let $K \subset \R^d$;
$int(K)$ will be  the interior of $K$,  $\pa K$ 
the boundary of $K$, $cl(K)$ or $\overline{K}$    the closure of $K$,
$K^c=\R^d \setminus K$.
 For every set $K \subset\mathbb{R}^d$, $co(K)$ is the convex hull of $K$.  The elements of $\R^d$ are called vectors, the zero vector of $\R^d$ is denoted by $o$.
 Let  $B(z,\rho)=\{x\in\R^d\,:\,|x-z|<\rho\}, 
S^{d-1}=\partial B(o,1) $             
and let $D(z,\rho) =cl(B(z,\rho))$.
The notations $B_\rho (x), D_\rho (x)$ will also be used respectively for open, 
closed balls of radius $\rho$ centered at $x$.
The usual scalar product between vectors $u,v\in \RR^d$ will be denoted by $\langle u,v\rangle $.
The closed segment with end points $x_1,x_2 \in \RR^d$ is denoted by $[x_1,x_2]$. 
A cone $C$ is a subset of $\R^d$ with the following property:
when $x\in C$, then $\forall \la >0$ $\la x \in C $. A closed cone $C$ contains its vertex $o$.  
The set $C\cap (-C)$ is the apex set of a closed cone $C$.
$C$ is a closed pointed cone if $C\cap (-C)=\{o\}$.

 Let $A$ be a body.
Let $q\in  A$; the {\em tangent cone } of $A$ at $q$  is defined in \cite{Federermeasure} as:
$$
\Tan(A,q)=\{v \in \RR^d: \forall \varepsilon > 0 \, ,
\exists  x\in A\cap B_\vare(q)\; , \exists r>0 \,\mbox{s.t.} \; |r(x-q)-v|< \vare\}.$$
Let us recall that if  $\Tan(A,q)\neq \{0\}$ then 
$$S^{d-1}\cap \Tan(A,q)= \bigcap_{\vare >0} cl(\{\frac{x-q}{|x-q|}:\, x\in A\cap B(q,\vare),  x\neq q\}).$$

The {\em normal cone } at
$q$ to $A$ is the non empty closed convex cone,  given by:
\begin{equation}\label{defNor}\Nor(A,q)=\{u\in\RR^d: \langle u,v\rangle \le 0 \quad
\forall v \in \Tan(A,q)\}.
\end{equation}

The {\em dual} cone of a cone  $K$ is
$$K^\star=\{y\in \RR^d : \langle y,x \rangle \geq 0 \quad \forall x \in K\}.$$
Thus 
\begin{equation}
 \label{nor=dualditan}
 \Nor(A,q)=-\{\Tan (A,q)\}^\star.
\end{equation}

Let   $A$  be  a body of
 $ \RR^d$ and $R>0$. Let
\begin{equation}\label{defA_R}
 A_R=\{x\in \RR^d : \dist(x,A) < R\}=\cup_{a\in A}B(a,R) ,   
\end{equation}
\begin{equation}\label{defA'_R}
 A'_R=(A_R)^c=\{x\in \RR^d : \dist(x,A) \geq R \}= \cap_{a\in A}( B(a,R))^c. 
 \end{equation}
\begin{definition}(\cite{Federer}) Let $A$ be a body. Let
$$Unp(A)= \{x\in \RR^d:   \mbox{there exists a unique point } \xi_A(x) \in A \mbox{\, nearest to } x \}.$$

	 If $A \subset \R^d,\, a\in A$, then $reach(A,a)$ is the supremum of all
	numbers $\rho$ such that for every $x\in B(a,\rho)$ 
	there exists a unique point $b\in A$ satisfying:
	$|b-x|=\dist(x,A)$:
	$$reach(A,a):= \sup \{\rho > 0: B(a,\rho) \subset Unp(A)\};$$
	and:
	$$reach(A):=\inf \{reach(A,a): a \in A \}.$$
	\end{definition}
	
\begin{definition}
For $a\in A$, let
$$ Q^{(a)}:= \{v\in \RR^d: \dist(a+v, A)=|v|\}.$$
\end{definition}	

Let us recall the following facts:
\begin{proposition}\label{prop2.1federer}\cite[Theorem 4.8,(2), (7) and (12)]{Federer}. Let $A$ be a body, $a\in A$, then

\item[i)]
$Q^{(a)} \subset \Nor(A,a).$
\item[ii)]
If $x\not\in A, x\in Unp(A), a=\xi_A(x)\in \pa A, reach(A,a)=\rho>0$, then 
$$A \subset(B(a+\rho\frac{x-a}{|x-a|}, \rho))^c.$$
\item[iii)]
If $reach (A,a)> r >0$, then 
$$
\Nor(A,a)=\{\la v: \la\geq 0, |v|=r, \xi_A(a+v) =a\};$$
 $\Tan(A,a)$ is the   dual cone of $-\Nor(A,a)$.
\end{proposition}

	\begin{definition}\label{htorto}
Let  $\; b_1, b_2 \in \R^d, | b_1 - b_2 | < 2R $ and let $  \mathfrak{h} (b_1,b_2)  $ be the intersection of all closed balls of radius $ R $ containing $   b_1, b_2  .$
\end{definition}

 \begin{proposition}(\cite[Theorem 3.8]{Colman}, \cite[Lemma 3]{Ratay2})\label{p1} 
The body	$ A $ has reach $ \ge R $ if and only if  for every $\; b_1, b_2 \in A, 0 <| b_1 - b_2 | < 2R $ the set $A \cap   \mathfrak{h} (b_1,b_2) $ is connected.

\end{proposition}
 
 \begin{remark}\label{rhullandrhulloid} The R-hull of a set $ E $ was introduced in  \cite[Definition 4.1]{Colman} as the minimal set $\hat{E}$ of reach
 $ \ge R $  containing $ E. $ Therefore, if $reach(A) \geq R$, then $A$ coincides with its $R$-hull. The R-hull of a set E may not exist, see \cite[Example 2]{Colman}. \end{remark} 
 
\begin{proposition}\cite[Theorem 4.4 and Theorem 4.6]{Colman}\label{theorem4.4} Let $A \subset \R^d$.
\begin{itemize}
\item[i)] If $reach (A'_R) \geq R$ then $A$ admits $R$-hull $\hat{A}$ and   
$$\hat{A}=(A'_R)'_R.$$
\item[ii)] If $A$ admits $R$-hull then $reach(A'_R)\geq R$.
\end{itemize}
\end{proposition}
 
Let us also recall the following result, see also \cite[Lemma 4.3]{Ratay1} :
 \begin{proposition}\cite[Theorem 3.10]{Colman} \label{prop3.10}
  Let $A\subset \R^d$ be a closed set such that $reach(A) \geq R > 0$.  If $D\subset \R^d$ is a closed set such that for every $a,b \in D$
  with  $|a-b|<2R$: 
  $$\mathfrak{h}(a,b)\subset D, \quad A\cap D\neq \emptyset$$
  holds, then $reach(A\cap D)\geq R$.
 \end{proposition}

\section{ $R$-bodies}\label{Rbodiessection}
 
 Next definitions  have been introduced in  \cite{Perkal} and in \cite{MLVRbodies}.
 
Let  $ R $ be a fixed positive real number. From now on
  $ B $ will be every open ball of radius $  R $,
    $B(x)$ will be the open ball of center  $ x \in  \RR^d$ and radius $  R $
and $D=cl(B), D(x)=cl(B(x))$.
 \begin{definition}\label{d1}
 Let $A$ be a body,  $ A $ will be called a R-body  if $\;  \forall y \in A^c$,  there exists an open  ball $ B $ in $\RR^d$,  satisfying: $ y\in B \subset  A^c. $ 
This is equivalent to say:  
 $$   A^c = \cup \{B : B \cap A = \emptyset \}; $$
 that is:
 $$   A= \cap \{B^c : B \cap A = \emptyset  \}  .$$
Notice that $\RR^d$ is  a $R$-body, since there are no points $y$ in its complement.
  \end{definition}
  
 \begin{definition}\label{d+1}
  Let $ E \subset \RR^d$ be a body with the property that there exists $B\subset E^c$. The  body:
  $$  co_R(E) :=  \cap \{B^c : B \cap E = \emptyset  \}  $$
 will be called the \textbf{R-hulloid} of $ E$, see \cite{MLVRbodies}. 
  If there are no balls $B \subset E^c$ then $co_R(E)=\RR^d$.
   \end{definition}

 \begin{remark} In \cite{Perkal} the sets defined in Definition \ref{d1} are called $2R$ convex sets and the sets defined in Definition \ref{d+1} are called $2R$ convex hulls. On the other hand  Valentine \cite[pp. 99-101]{val} and Fenchel \cite[p.42]{Fenc} use the name  $R$-convex sets for convex sets with special properties depending on $R$.  To avoid misunderstandings  we decided in \cite{MLVRbodies} to call $R$-bodies and $R$-hulloids   the sets defined in Definition \ref{d1} and in Definition \ref{d+1} respectively.
   \end{remark}
  \begin{remark}\label{R-hull=R-hulloid} The R-hulloid of a bounded set always exists. Let us notice that $ co_R(E)$ is a $R$-body (by definition) and $ E \subset   co_R(E) .$
  Moreover $ A$ is a $R$-body if and only if  $ A =   co_R(A)$.
 
   \end{remark}  
  Clearly every  convex  body $E $ is a $R$-body (for all positive $R$) and its convex hull $co(E)=E$ coincides with its $R$-hulloid.

 \begin{proposition}\cite[Theorem 3.11 and 3.10]{MLVRbodies} \label{RbodiesinH} For $d \geq 2$, every closed  non empty subset  of a  affine linear proper subset of $\RR^d$ is a  $R$-body;   closed subsets of  the boundary of a ball of radius greater or equal than $R$ are $R$-bodies too. 
 \end{proposition} 
 
 \begin{remark}\label{remcl(cup)=co(A)} 
 Perkal  in \cite{Perkal} proved that, if $A$ is a body and $int(A)\neq \emptyset$, then:
  \begin{equation}\label{cl(cup)=co(A)}
 cl(\bigcup_{r  > R}co_r(A))=co(A).
 \end{equation}
   Walther  (\cite{Wal})  claimed that if $A$ is a body and $ int(co(A)) \neq \emptyset$, then \eqref{cl(cup)=co(A)} holds.
 \end{remark}
 If $int(co(A))=\emptyset$, equality \eqref{cl(cup)=co(A)} may   not be true:  let $A$ be a body  not connected subset of an affine linear proper  subset of $\RR^d$. Then, by Proposition \ref{RbodiesinH}, for all positive $r$, the $r$-hulloids 
 $co_r(A)=A$ and $A \neq co(A)$; thus  \eqref{cl(cup)=co(A)}  does not hold.

 \begin{proposition}\label{precedentiproprieta}
 Let $E$ be a non empty set.
    The following facts have been proved in \cite{Perkal}.
  
 \begin{itemize}
 
 \item \textbf{a} $\quad   co_R(E) =( E'_R)'_R; $
  \item \textbf{b} \; Let $ A^{(\alpha)} , \alpha \in \mathcal{A}$
  be R-bodies, then $\cap _{\alpha  \in \mathcal{A}}$ $A^{(\alpha)} $ is an R-body; 
\item  \textbf{c} \; $co_R(E)\subset co(E)$ for all $R> 0$;
\item  \textbf{d} \;  $0< R_1 \leq R_2 \Rightarrow co_{R_1}(E) \subseteq co_{R_2}(E)$.
 \end{itemize}
   \end{proposition}
 It is easy to prove also the following:
  \begin{itemize}
 \item \textbf{e}  \; $\, \pa E \subset \pa co_R(E) $;
  \item  \textbf{f} \; $int(E)\subset int(co_R(E))$ for all $R> 0$.
 \end{itemize}

In \cite[Theorem 3.4]{MLVRbodies} the following fact was proved:
\begin{equation}\label{formulacoR(E)}
 co_R(E)=E_R\cap \Big(\pa(E_R)\Big)_R'.
\end{equation}

Moreover:
 \begin{theorem}\label{inviluppoR-conigeneral}
Let $E$ be a body. Then:

\begin{equation}\label{formulainviluppoR-conigeneral2*}
(\pa E_R)'_R =co_R(E)\bigcup  E'_{2R}.
\end{equation}
\end{theorem}  
\begin{proof}

Let us prove  \eqref{formulainviluppoR-conigeneral2*}. This is equivalent to prove that:
\begin{equation}\label{formulapaE'R}
(\pa E_R)_R =(co_R(E))^c\bigcap  E_{2R}.
\end{equation}
By (\ref{defA_R})    with  $A=\pa E_R $, it follows that: 
 $$\Big(\pa E_R\Big)_R=\cup \{B(x): x\in \pa(E_R)\}.$$
Then, as $\pa E_R
=\{x : \dist (x,E)=R\}$:
\begin{equation}\label{secondformulapag8*}
(\pa E_R)_R= \bigcup \{ B(x):  \dist(x,E)=R \}.
\end{equation}
Therefore:
\begin{equation}\label{secondformulapag8*conapice}
(\pa E_R)'_R= \bigcap \{B^c(x):  \dist(x,E)=R\}. 
\end{equation}
 From  \eqref{secondformulapag8*}, it holds:
 \begin{equation}\label{thirdformulapag8*}
(\pa E_R)_R= 
\end{equation} 
 $$ \Big( \cup \{B(x):   \dist(x,E)\geq R\}\Big) \bigcap
 \Big( \cup \{B(x):   \dist(x,E) \leq  R\}\Big).$$

Since: 
$$E_{2R}= E+B(o,2R)=(E+B(o,R))+B(o,R)=E_R+B(o,R),$$
 then:
 \begin{equation}\label{E2R'}
E_{2R}= \cup \{B(x): \,  \dist(x,E) \leq R\}.
\end{equation}
By Definition \ref{d+1}:
\begin{equation}\label{coR(E)compl}
(co_R(E))^c=\cup \{B(x): \,  \dist(x,E)\geq R\}.
\end{equation}
 \eqref{formulapaE'R} follows   from \eqref{thirdformulapag8*} and last two equalities.
\end{proof}

\subsection{$R$-supporting  balls and $R$-supported bodies}\label{Rsupportsection}

\begin{definition} \label{defRsupporting} Let $A$ be a body. Let $a\in \pa A$. Let $v\in S^{d-1}$.
 We say that the ball $B(a+Rv)$  is $R$-{\em supporting} $A$ at $a$ if:
 $$A\subset (B(a+Rv))^c;$$
 $v$ will be called  a   unit vector $R$-supporting $A$ at $a$.

  In other words, $v\in S^{d-1}$ is a unit vector $R$-supporting $A$ at $a$ if and only if for all $b\in A$  
\begin{equation}\label{4*}
\langle v,a-b\rangle \geq -\frac{|a-b|^2}{2R}.
\end{equation}
\end{definition}

In \cite[\S 2.2]{Gol} a closed ball $D_R=cl(B)$ of radius $R$ is defined an outer support (closed) ball  of $A$ if $D_R\cap A\neq\emptyset$ and $D_R\cap int(A)=\emptyset$. The closure of a $R$-supporting ball $B$ is an outer closed support ball $D_R$ and conversely.

Golubyatnikov and Rovensky \cite{Gol} 
have defined the class 
 $\mathcal{K}_2^{1/R}$  of bodies  $A$, with non empty interior, satisfying the following property:
 \begin{equation}\label{defK21/R}
 \forall x\in A^c \mbox{\, there exists a closed ball \, } D_R\ni x: D_R\cap int(A)=\emptyset.
 \end{equation}
 
The class of $R$-bodies, is strictly contained in the class  $\mathcal{K}_2^{1/R}$, see \cite[Theorem 6.1]{MLVRbodies}. 
\begin{definition} If $a\in \pa A$, let us denote:
\begin{equation}\label{defvR-supporting}
\mathcal{N}_R(A,a) :\equiv \{ v \in S^{d-1}: A \subset (B(a+Rv))^c\}
\end{equation}
 the set of unit vectors $R$-supporting $A$ at $a$.
\end{definition}
 Cuevas and others in \cite[Proposition 2]{Cue}  proved that for a $R$-body $A$ and for $a\in \pa(A)$ the set of $R$-supporting unit vectors $\mathcal{N}_R(A,a)$ is not empty; they also  
call a body $A$ a $R$-rolling set, if for every $a\in \pa A$
there is a $R$-supporting ball B of $A$; therefore $\mathcal{N}_R(A,a)$ is non empty.
Since the $R$-rolling set name is  used  in \cite{Wal} with another  meaning (see \S \ref{R-bodies,reach,R-rolling sets}), here a different name has been used.
 \begin{definition} A body  $A$ is called a $R$-supported body if for every  point $a \in \pa A$ the set  $\mathcal{N}_R(A,a)$ of the directions $R$-supporting $A$
 is non empty. 
\end{definition}
From \cite[Theorem 3.6]{MLVRbodies}  a more general result follows:
\begin{proposition}\label{corollary le matematiche}
If $E$ is a body and $A=co_R(E), a\in \pa A$ then there exists a ball $B\subset A^c$ , with $ a\in \pa B$; moreover $\pa B \cap \pa E\neq \emptyset$.
\end{proposition}
\begin{remark}\label{viceversaEggleston} Proposition  \ref{corollary le matematiche} implies  that, if $A$ is a $R$-body, then for all $a\in \pa A$ the set of $R$-supporting unit vectors $\mathcal{N}_R(A,a)$ is non empty. Then $A$ is a $R$-supported body, The converse of this fact it is not true, see Proposition 2 in \cite{Cue}. Other examples: let $V$ be the set of the vertices of an equilateral triangle $T$ with circumradius less than $R$.  At each   $v\in V=\pa V$, $\mathcal{N}_R(V,v)$ is non empty;  $V$ is not a $R$-body since the center of $T$ lies in $V^c$ but does not belong to an open disk of radius $R$, avoiding $V$.
In \cite[formula (16), Theorem 5.7]{MLVRbodies} it has been defined 
 a body $E$ with non empty interior, with a $R$-supporting ball at every point of its boundary  which is not a $R$-body.  
\end{remark}
The following facts hold:
\begin{itemize}
\item[a)] the family of $R$-bodies is a proper subset of the family
of $R$-supported bodies;
\item[b)] every closed subset of the boundary of a $R$-body is a
$R$-supported body;
\item[c)]the intersection of two $R$-supported bodies is a $R$-supported body;
\item[d)] if a $R$-body $A$ is not connected, then the closed connected components of $A$ are $R$-supported bodies.
\end{itemize}

\begin{lemma}\label{proprieta'supportedbodies}Let $E$ be a $R$-supported body and $A=co_R(E)$, then:
\begin{itemize}
\item[i)] \, $\pa E \subset \pa A$;
\item[ii)] \, $int (E) \subset int(A)$;
\item[iii)]\, $  \mathcal{N}_R(E,a)= \mathcal{N}_R(A,a)\,\, \forall a\in \pa E$;
\item[iv)] \, $co_R(E) = E_R \cap \{(B(x))^c: x= a+ R\theta, a\in \pa E, \theta \in  \mathcal{N}_R(E,a)\}$.
\end{itemize}
\end{lemma}
\begin{proof}
Let $a\in \pa E$. Let $\theta \in  \mathcal{N}_R(E,a)$, then 
$$a\in \pa B(a+R\theta) \mbox{\quad and \quad } B(a+R\theta) \cap E = \emptyset.$$ By definition of $A=co_R(E)$, $B(a+R\theta) \subset A^c$. Therefore $a\in \pa A$ and  $\theta \in \mathcal{N}_R(A,a)$. Then, i)
is proved  and $  \mathcal{N}_R(E,a) \subseteq \mathcal{N}_R(A,a)$.
Since $E\subset A=co_R(E)$, then ii) is obvious.
To prove iii),  let  $a\in \pa E$; then for $\theta \in \mathcal{N}_R(A,a)$ it holds $B(a+R\theta)\cap A= \emptyset$. Since $E\subset A$ then $B(a+R\theta)\cap E= \emptyset$. Then, $\theta \in \mathcal{N}_R(E,a)$
and equality in iii) is proved.
  
 From \eqref{formulacoR(E)}, \eqref{secondformulapag8*conapice} and
 $$\pa E_R= \{x: dist(x,E)=R\}= \{a+R\theta: a \in \pa E, \theta \in  \mathcal{N}_R(E,a)\},$$
  iv) follows.
\end{proof}

\subsection{ $R$-supporting balls to the  $R$-hulloid of a body}\label{Rsupportingsection+}

  \begin{definition}\label{defconvonS}
 Let $S$ be a sphere in $\RR^d$ of radius $\rho > 0$ centered at the origin; there is a one-to-one map between closed cones in $\RR^d$ and closed subsets of $S$.
 For every closed cone $K$ of $\RR^d$ let $\mathcal{K}=K \cap S$;  let $\mathcal{K}$ be a closed  subset of $S$,
 let $K=\{\la v: v\in \mathcal{K}, \la \geq 0 \}$ the related cone in $\RR^d$.

 For  $\mathcal{K} \subset S$ let us define  the spherical convex hull
$$co^S_{sph}(\mathcal{ K}):\equiv co(K) \cap S.$$
 \end{definition}
 A closed convex  cone $C$  is pointed if: 
$$ap(C):=C\cap (-C)=\{o\}.$$ 

Last definitions have similar extensions for a sphere $S$ centered at every point $c\in \R^d$, not necessarily at the origin $o$    and  for cones with  vertex $c\neq o$.

Let us introduce the following notations in this section:

Let $\,  H:=\{x\in \R^d: \langle x,v \rangle =0\}$, with
 $v\in  S^{d-1}$,
$$H^+:=\{x\in \R^d: \langle x, v\rangle >0\}, \, H^-:=\{x\in \R^d: \langle x, v\rangle <0\};$$
$$\mathbf{H}^+= \{x\in \R^d ; \langle x,v\rangle \geq 0\}, \,
\mathbf{H}^-:=\{x\in \R^d: \langle x, v\rangle \leq 0\};$$  
 As in \cite[\S 1.3]{Schn}, let us define 
 $A$ and $B$ strictly separated by $H$ if
 $$A \subset H^+, \quad  B \subset H^-$$
 or conversely.

 Let $E$ be a body, $a\in \pa co_R(E)$, $a\not \in E$,  $B(o,R)$ an open ball  $R$-supporting $co_R(E)$ at $a$. With no restrictions it can be assumed that $B$ is centered at $o$. 
 
  By Proposition \ref{corollary le matematiche}, $\pa B \cap E\neq \emptyset$. Let $S=\pa B$.
 
 \begin{lemma}\label{lemmapaolosupporting} Let $E$ be a body, $a\in \pa co_R(E)$. Let $B(o,R)$ be  an open ball  $R$-supporting $co_R(E)$ at $a$. Let $S=\pa B(o,R)$ and let   $F$ be the cone related to 
  $\mathfrak{F} =  E\cap S$.
   Assume that $C=co(F)$ is pointed and  $a\not \in C $.  Then:
  \begin{itemize}
  \item[i)] there exists $v\in S^{d-1}$ and $H:=\{x\in \R^d: \langle x,v \rangle =0\}$ so that $H$ strictly separates $\{a\}$ from $\mathfrak{F}$, that is:
  $$\{a\} \subset H^+, \quad \mathfrak{F} \subset H^- ;$$
  \item[ii)] let $t>0, B_t=\{x:|x-tv|< R\}$. If  $t$ is sufficiently small, then
  \begin{equation}\label{AinBtFnotinBt}
  a \in B_t,  \quad B_t\cap E =\emptyset.
    \end{equation}
\end{itemize}   
   \end{lemma}
  \begin{proof}  Let us consider $C= co(F)=co(\mathfrak{F}\cup \{o\})$ and let $A=co(\{a,o\})$, then $ap(A)=\{o\}, ap(C)=\{o\}$. Since $a\not \in C$, then $A\cap C=\{o\}$.
  A result of Klee \cite[Theorem 2.7]{Klee}, in $\R^d$,  see also \cite[Theorem 4.2]{Soltan}, states that for two closed convex cones $A$ and $C$ satisfying  the above conditions there exists a hyperplane $H:=\{x\in \R^d: \langle x,v \rangle =0\}$, which {\em sharply} separates $A$ and $C$, that is:
  $$A\setminus ap(A) \subset H^+, \quad C\setminus ap(C) \subset H^-.$$
    So i) is proved.

    Let us prove now ii).

  Let $\cos \phi= \langle \frac{a}{|a|}, v \rangle >0 $
 since $a \in  H^+$; if $0<t< 2R\cos \phi$ then  
  $$|a-tv|^2=R^2+t^2 |v|^2-2Rt\cos \phi < R^2, $$
 so  $a\in B_t$.
  
 Let us show that $B_t\cap E =\emptyset$. 
 The open ball $B_t=(B_t\cap B) \cup (B_t\setminus B)$ is union of  two non empty  sets, for $ 0 < t$ small enough.  Since  $B\cap E=\emptyset  $ by assumption, then $(B_t\cap B) \cap E =\emptyset$.
   Moreover $\overline{B_t} \setminus B \subset H^+$ for $t>0$ and 
    for $t \to 0^+$ the set $\overline{B_t} \setminus B \to  S \cap \mathbf{H}^+ $.   
 By i)  $\mathfrak{F} \subset H^-$, thus $S\cap\mathbf{H}^+$ and $E$ are disjoint closed sets; 
 then  
  for $t>0 $ small enough, $ B_t \setminus B $ and $E$ are disjoint  sets too.
 Then $B_t$ and $E$ are disjoint sets 
 for $t>0$ small enough.
   \end{proof}

\begin{theorem}\label{theorem1suportingball} Let $E$ be a body in $\RR^d$. Let   $a \in \pa co_R(E)\setminus E$. Then, there exists a ball $B$,   
  $R$-supporting $co_R(E)$ at $a$. Let $S=\pa B$ (with no restrictions it can be assumed that $B$ is centered at $o$).  Then 
  \begin{itemize}
  \item[i)] $\pa E \cap S $ contains at least two points.
    \end{itemize}
Let  $\mathcal{F}=\pa E \cap S$  and   $F$ be the related   cone and $C=co(F)$, then
 \begin{itemize}
  \item[ii) ]if $C$ is  a pointed  cone, then  there exist distinct points  
 $x_1, \ldots, x_s \in \pa E \cap S$, $2\leq s\leq d$, , such that 
 \begin{equation}\label{xincosphproper}
a \in co^S_{sph}(\{ x_1, \ldots, x_s\}).
\end{equation}
  \end{itemize}
    \end{theorem}
 
   \begin{proof}  By Proposition \ref{corollary le matematiche}, if $a\in \pa  co_R(E)\setminus E$, there exists at least a support ball $B\subset (co_R(E))^c \subset E^c$ such that  $a\in \pa B=S$ and  $ \mathfrak{F}=\pa E \cap S$ is a non empty closed set. Let us prove i). By contradiction, let 
$\mathfrak{F}=\{z\}$. Since $a\not \in E$, $a\neq z$, then by Lemma \ref{lemmapaolosupporting}, if $t>0$ is sufficiently small, 
$a\in B_t\subset E^c$, impossible as 
$a \in co_R(E)$.

Let us prove ii). If $C$ is a pointed cone, by ii) of  Lemma \ref{lemmapaolosupporting}, if $a\not \in C$, there exists $B_t\ni a$, $B_t \cap E=\emptyset$; impossible since $a \in co_R(E)$. Thus $a \in C$.
 By  \cite[Theorem 7  pag. 12 ]{Fenc}  $a$ is a linear combinations of vectors of $F$, then \eqref{xincosphproper} follows.
 \end{proof}
 
 \begin{remark} Under the notations  of Lemma \ref{lemmapaolosupporting}, in case $C$ is not a pointed cone, let $r=dim(ap(C))$. Then  \cite[Theorem 7]{Fenc} (for convex cones) proves that  there exist $1\leq  r \leq d $  distinct points
$x_1, \ldots, x_{r+1} \in \pa E \cap S$,    such that \begin{equation}\label{xincosphnotproper}
\bigcup_{\la_1, \la_2,\ldots , \la_{r+1} \geq 0}\{\sum_{i=1}^{r+1}\la_i x_i\}=C\cap(-C).
\end{equation}
  \end{remark}
Let us notice  that \eqref{xincosphnotproper}, with $r=1$, implies that there are two opposite points $x_1,x_2$ on the spherical  surface $S$. When $r=2$ there are three points $x_1,x_2,x_3$ such that the circumradius of 
$co(\{x_1,x_2,x_3\})$ is equal to $R$.

\section{ $R$-supported bodies vs sets of reach $\geq R$ and $R$-rolling sets}\label{R-bodies,reach,R-rolling sets}

\begin{lemma} \label{lemmainclusionnormalsets}
Let $A$ be a body, $a \in \pa A$.  Let $\mathcal{N}(A,a)=\Nor(A,a)\cap S^{d-1}$, then

  \begin{equation}\label{N=N_R}
\mathcal{N}_R(A,a)\subseteq\mathcal{N}(A,a). 
\end{equation}

\end{lemma}
\begin{proof} If $\mathcal{N}_R(A,a)\neq \emptyset$,
let   $B(a+Rv)$ be $R$-supporting $A$ at $a$, then 
 $ A \subset (B(a+Rv))^c$. Let 
  $$w=\lim_{z_n \in A, z_n \to a } \frac{z_n-a}{|z_n-a|}\in \Tan(A,a);$$
  by \eqref{4*}, with $b=z_n$  the inequalities
 $$\langle v, \frac{z_n-a}{|z_n-a|} \rangle \leq \frac{|z_n-a|}{2R}$$
 hold, $\forall n \in \N.$ Thus  the inequality
 $ \langle v,w \rangle \leq 0$ holds. Then
 $v \in \Nor(A,a).$
\end{proof}

\begin{theorem}\label{N=NRequality} Let $A$ be a body, $a \in \pa A$. If $reach(A,a)\geq R$ , then 
$$\mathcal{N}(A,a) = \mathcal{N}_R(A,a).$$
\end{theorem}
\begin{proof}Let $w\in  \Nor(A,a)$, $0 < r < reach(A,a)$; then by iii) of Proposition \ref{prop2.1federer}, 
$$\xi_A(a+rw)=a .$$
Therefore, by ii) of Proposition \ref{prop2.1federer},  $B_r(a+rw)$ is $r$-supporting $A$ at $a$; then by \eqref{4*} for all $ z\in A\setminus\{a\}$ the inequalities
  \begin{equation}\label{ineqRsupporting}
   \langle w, \frac{z-a}{|z-a|} \rangle \leq \frac{|z-a|}{2r}
  \end{equation}
hold.
By a limit argument for $r\to R^-$ then \eqref{ineqRsupporting} holds for $r=R$ 
and the unit vector $w$ is $R$-supporting $A$ at $a$, so $w \in \mathcal{N}_R(A,a)$. This proves that 
$$\mathcal{N}(A,a) \subseteq \mathcal{N}_R(A,a).$$
The opposite inclusion follows from \eqref{N=N_R}.
\end{proof}

 \begin{remark}\label{reach>=RimpliesRbodies} It was noticed in \cite[Corollary 4.7]{Colman} and proved in \cite[Proposition 1]{Cue} that, when the R-hull exists,  it coincides with the R-hulloid. If $A$ has $reach$ greater or equal than $R$, then  (see Remark \ref{rhullandrhulloid}) $A$ has $R$-hull, which coincides with $A$ and with its $R$-hulloid, and $A$ is a $R$-body; then the class of $R$-bodies contains the class of sets of reach greater or equal than $R$.
The family of the sets of reach greater or equal than $R$ is not closed with respect to the intersection.
\end{remark}

\begin{remark}\label{minimalfamily}Let $\mathcal{R} $ be the family of the sets of reach $ R. $
The family of the $R$-bodies is the minimal
family containing $\mathcal{R} $ and closed with respect to the
intersection.
\end{remark}
\begin{proof}
Let  $\mathcal{F} $ be a family closed with respect to the intersection, such that $\mathcal{F} \supset \mathcal{R} $.
Let $ A $ be a $R$-body, then $ A = \cap \{B^c : B \cap A =
\emptyset\}$. As $ reach (B^c) = R$ for every $B^c \supset A$, then  $A$ is 
intersection of sets of reach $ R $, which are in the family.  Then $A \in \mathcal{F}$. 
\end{proof}

Next theorem gives  necessary and sufficient conditions for a $R$-supported body to have reach greater or equal than $R$.

  \begin{theorem}\label{RbodyfrontieraPaolo}Let $d\geq 2$, $A$ be a body in $\RR^d$. Then the following properties are equivalent:
   \item[i)] $A$ is a $R$-supported body and 
   \begin{equation}\label{N(A,a)subsetN_R(A,a)}
   \Nor(A,a)\cap S^{d-1}= \mathcal{N}_R(A,a) \quad\forall a \in \pa A;
   \end{equation}
   
   \item[ii)] $reach(A)\geq R.$
   \end{theorem}
  \begin{proof}If $reach(A)\geq R$, from Theorem \ref{N=NRequality}  equality holds in \eqref{N=N_R}; from Remark  \ref{reach>=RimpliesRbodies},  $A$ is a $R$-body, then it is a $R$-supported body. Then ii) implies i).
  
 By contradiction let us assume that  i) holds and $reach(A)< R$. Then there exists $a \in \pa A$, $x\in B(a)\cap A^c$, $x\not \in Unp(A)$. Thus there exists $a_1\in \pa A$, $a_1\neq a$ satisfying:
  $$\dist(x,A)=|x-a|=|x-a_1|=R_1 < R.$$
  Let $v=x-a$, then 
  $$\dist(a+v,A)=\dist(x,A)=|x-a|=|v|,$$
  so $v \in Q^{(a)}$. By i) of Proposition \ref{prop2.1federer}, $v\in \Nor(A,a)$. Then  $\theta=v/|v|$ is a unit vector  $R$-supporting $A$ at $a$; that is $\theta\in \mathcal{N}_R(A,a)$, thus $A\subset (B(a+R\theta))^c$. Since $x-a=R_1\theta $,  
  $R_1< R$
  and $\theta$ is the inner unit normal to $B(a+R\theta)$ and to $B(a+R_1\theta)$ at $a$, then $B(a+R_1\theta,R_1 )\subset B(a+R\theta)$.
  Then $a_1\in A\cap B(a+R\theta)$, so $\theta\not \in  \mathcal{N}_R(A,a)$. Contradiction.
  \end{proof}
 
\begin{example}\label{examplenotreach}  The following sets are examples of $R$-bodies where strict inclusion holds in \eqref{N(A,a)subsetN_R(A,a)}; then,  they have reach less than $R$: 
\begin{itemize}
\item[a)] $A=\{a,b\}$, with $|a-b| <R$;
\item[b)]  $H\setminus B(o,r)$, with $r< R$, $H$ plane in $R^3$, through $o$  
 (see Proposition \ref{RbodiesinH});
\item[c)]  in \cite[fig.1(a)]{Cue} there is  a $R$-body with non empty interior and strict inequality holds in  \eqref{N(A,a)subsetN_R(A,a)}.
\end{itemize}
\end{example}

  \begin{lemma}\label{lemmaWalther}   
  Let $A$ and $\overline{A^c}$ be $R$-supported bodies. If:
 \begin{equation}\label{paAinpacclosure1}
 a_0 \in \pa A \cap \pa \overline{A^c},
 \end{equation}
 then there exist two open balls $B(y_0),\, B(y_1)$ satisfying:
 \item[i)]$B(y_0) \subset A^c,B(y_1)\subset (\overline{A^c})^c $;
 \item[ii)] $B(y_0)\cap B(y_1)=\emptyset, \{a_0\}=\pa B(y_0)\cap \pa B(y_1) .$
   \end{lemma}
 \begin{proof} For every $A\subset \RR^d$ it holds $\RR^d=int(A)\cup \pa A \cup A^c$,  $\overline{A^c}=\pa A \cup A^c$, and 
  \begin{equation}\label{int(A)=compl(cl(compl(A))}
  int(A)= (\overline{A^c})^c.
  \end{equation}
 As $a_0 \in \pa A$ and $A$ is a $R$-supported body,   there exists an open ball $B(y_0) \subset A^c$, where $a_0\in \pa B(y_0)$. By assumptions \eqref{paAinpacclosure1}, $a_0 \in  \pa \overline{A^c}$ too and  $  \overline{A^c}$ is a $R$-supported body, then 
 there exists an open ball $B(y_1) \subset (\overline{A^c})^c$, $a_0 \in \pa B(y_1)$ and i) is proved. 
 
Moreover  by \eqref{int(A)=compl(cl(compl(A))} the two open balls $B(y_0),B(y_1)$  also satisfy  ii).\end{proof}
 \begin{theorem}\label{RbodiesWalther} Let $A$ and $\overline{A^c}$ 
 be $R$-supported bodies  and 
 \begin{equation}\label{paAinpacclosure}
 \pa A = \pa \overline{A^c}
 \end{equation}
then,  $reach(A) \geq  R$.
   \end{theorem}
 \begin{proof} Let us assume, by contradiction that $reach(A) < R$. Then there exists $ a_0\in \pa A$, so that $reach(A,a_0) < R$. Then $\exists \, r_1 <R, \exists \, x_1 \in A^c $ and $a_1\in \pa A \setminus \{a_0\}$ satisfying:
 $$r_1=\dist(x_1,A)=|x_1-a_0|=|x_1-a_1|< R, $$
 \begin{equation}\label{B(x1,r1)inAc}
  B(x_1,r_1) \subset A^c.
 \end{equation} By assumption \eqref{paAinpacclosure}, $a_0\in \pa \overline{A^c}$ and 
 by Lemma \ref{lemmaWalther} there exist two open balls $B(y_0), B(y_1)$ satisfying: i),\, ii).
 Since $a_0\in \pa B(x_1,r_1)$ and inclusion \eqref{B(x1,r1)inAc} holds  then,  $B(x_1,r_1) \subset B(y_0,R)$ and
 $$a_1\in \pa  B(x_1,r_1) \setminus \{a_0\}\subset B(y_0,R)\subset A^c.$$
So $a_1\in A^c $, in contradiction with $a_1 \in \pa A$.  \end{proof}
 \begin{remark} Let $A$ be a body, $X=A^c$; as $\pa \overline{X} \subset \pa X$, then: 
 $$ \pa\overline{A^c} \subseteq \pa A^c=\pa A.$$
Last inclusion can  be strict, see also next example.
 \end{remark}
 \begin{example} If \eqref{paAinpacclosure} does not hold, then 
 it can be $reach(A) < R$. As example let  $|x_0-x_1| > 3R, |x_1-x_2| < R/4 $ and let
 $$A= D(x_0) \cup \{x_1\}\cup \{x_2\}.$$ 
 Then $A$ is a $R$-body and  
 $\overline{A^c}$ is a $R$-body too. But  $reach(A) =|x_1-x_2|/2 < R$.
 \end{example}
 
 \begin{remark}  Walther (\cite{Wal})  considered the class of non empty, path connected, compact sets $A$,  with the following property: a ball of radius $R$ rolls freely in $A$ and in $ \overline{A^c}$. These path connected compact sets $A$ are such that  $A$ and $\overline{A^c}$ 
 are $R$-bodies (\cite[Theorem 1]{Wal}). 
 \end{remark}
  Walther proved  several other properties of the boundary of these  sets  $A$:

 \begin{proposition}\cite[formula (33)]{Wal} Let $A$ and $\overline{A^c}$ 
 be $R$-bodies  and $A$ path connected compact set, then \eqref{paAinpacclosure} holds.
  \end{proposition}

 From last proposition and Theorem \ref{RbodiesWalther}, for  the Walther's $R$-rolling sets    a property, not stated in \cite{Wal}, can be  proved.
 \begin{theorem}\label{CorWal} Let $A$ and $\overline{A^c}$ 
 be $R$-bodies  and let $A$ be  a path connected compact set, then 
 $$reach(A) \geq R.$$ 
 \end{theorem}

 \section{$R$-cones}\label{section R-cones}
Let us introduce  a class of  $R$-bodies, that will be called $R$-cones.
 As the $R$-bodies are a generalization of convex sets, the $R$-cones are a generalization of convex cones.

\begin{definition}\label{curvedconedef}
 Let $\mathcal{K}$ be a body in  $S^{d-1}$. A {\em  $R$-cone} with vertex $o$   (more simply an $R$-cone)
  is the  $R$-body: 
 \begin{equation}
  \label{defunboundcone}
   C_{\mathcal{K}}:=\bigcap_{v \in \mathcal{K}} (B(Rv,R))^c.
   \end{equation}

Let  $x\in \RR^d$, an   $R$-cone with  vertex $x$ is the $R$-body:
 \begin{equation}\label{defunboundedcurvedcone}
  C_{\mathcal{K}}^x:=x+ C_{\mathcal{K}}=\bigcap_{ v\in \mathcal{K}} (B(x+Rv,R))^c.
 \end{equation}
  \end{definition}
  
  Let us notice that $ x\in \pa C_{\mathcal{K}}^x$ and the unit vectors in $\mathcal{K}$ are $R$-supporting  $C_{\mathcal{K}}^x$ at its vertex $x$, that is  
  \begin{equation}\label{KinNRK}
  \mathcal{K}\subset \mathcal{N}_R(C_{\mathcal{K}}^x,x).
  \end{equation}
Moreover
 \begin{equation}\label{K1inK2}
  \mathcal{K}_1\subset \mathcal{K}_2 \Rightarrow C_{\mathcal{K}_1}^x\supset  C_{\mathcal{K}_2}^x.
  \end{equation}
 \begin{theorem}\label{inviluppoR-conigeneral*}
Let $E$ be a $R$-supported  body. Then 
\begin{equation}\label{secondformulapag8}
\bigcap_{a\in \pa E}\mathcal{C}^a_{\mathcal{N}_R(E,a)}= (\pa E_R)_R'.
\end{equation}
\end{theorem}  
\begin{proof}
By definition:
\begin{equation}\label{CN_R(E,a)}
\mathcal{C}^a_{\mathcal{N}_R(E,a)}=\bigcap_{ v \in \mathcal{N}_R(E,a)} (B(a+Rv))^c =\bigcap_{  (B(a+Rv))^c\supset E} B(a+Rv))^c.
\end{equation}
Let us notice that $a\in \pa E, v \in \mathcal{N}_R(E,a)$ if and only if:
$$x=a+Rv, \mbox{\quad satisfies \quad}\dist(x,E)=R.$$
Thus
\begin{equation}\label{intersectionCN_R(E,a)}
\bigcap_{a\in \pa E}\mathcal{C}^a_{\mathcal{N}_R(E,a)}= \bigcap_{x :\, \, \,  \dist(x,E)=R} (B(x))^c.
\end{equation}
From \eqref{secondformulapag8*conapice} of Theorem \ref{inviluppoR-conigeneral}
the formula \eqref{secondformulapag8} follows.
\end{proof}

From  \eqref{secondformulapag8} and   \eqref{formulacoR(E)} we have:

\begin{corollary}\label{formulainviluppoR-conigeneral}
Let E be a $R$-supported  body. Then:
\begin{equation}
co_R(E)= E_R\bigcap \Big(\bigcap_{a\in \pa E}\mathcal{C}^a_{\mathcal{N}_R(E,a)}\Big).
\end{equation}
\end{corollary}

\begin{corollary}\label{inviluppoR-coni}
Let $A$ be a  $R$-supported body. Then $A$ is a $R$-body if and only if:  
\begin{equation}\label{formulainviluppoR-coni}
A= A_R\bigcap\Big( \bigcap_{a\in \pa A}\mathcal{C}^a_{\mathcal{N}_R(A,a)}\Big).
\end{equation}
\end{corollary} 
\begin{proof} If $A$ is a $R$-body then $A=co_R(A)$ and \eqref{formulainviluppoR-coni} follows from Corollary \ref{formulainviluppoR-conigeneral} with $E=A$. Conversely if \eqref{formulainviluppoR-coni}
holds for a $R$-supported body $A$, then by  \eqref{secondformulapag8} and \eqref{formulacoR(E)}, with $A$ in place of $E$, it follows that  $A=co_R(A)$.
\end{proof}

\subsection{$R$-cones, Tangent cones and Normal cones}\label{subsectioncones}

For simplicity, let us denote
 $\mathcal{N}_R(C_{\mathcal{K}})=\mathcal{N}_R(C_{\mathcal{K}},o)$, ${N}(C_{\mathcal{K}})=\Nor(C_{\mathcal{K}},o)$  and $\Tan(C_{\mathcal{K}})=\Tan(C_{\mathcal{K}},o)$.
  \begin{lemma}\label{Tan(A)inC_N_R(A)}Let $A$ be a  body, $a\in \pa A$ then

 \begin{equation}\label{TaninCNR}
   \Tan(A,a) \subset C_{\mathcal{N}_R(A,a)}.
\end{equation} 
 \end{lemma}
 
 \begin{proof} Since for every $ v \in \mathcal{N}_R(A,a)$, the inclusion $A \subset (B(a+Rv))^c$ holds, then
$$\Tan (A,a)\subset \Tan((B(a+Rv))^c,a).$$ 
Moreover
$$a+\Tan((B(a+Rv))^c,a)\subset (B(a+Rv))^c.$$
By previous inclusions:
$$a+\Tan (A,a)\subset \cap_{ v \in \mathcal{N}_R(A,a)}(B(a+Rv))^c $$
and by  \eqref{defunboundedcurvedcone}:
$$ \cap_{ v \in \mathcal{N}_R(A,a)}(B(a+Rv))^c=a+ C_{\mathcal{N}_R(A,a)}.$$
 Therefore:
 $$a+\Tan (A,a) \subset a+ C_{\mathcal{N}_R(A,a)}$$
and \eqref {TaninCNR} is proved.
\end{proof}

 \begin{lemma}\label{tanC_KinC_Kthm}
 Let $\mathcal{K}$ be a body in  $S^{d-1}$ and let  $C_{\mathcal{K}}$ be its related $R$-cone with vertex $o$. If $B(Rv)$ is a $R$-supporting ball to $C_{\mathcal{K}}$ at $o$, then $v\in \mathcal{K}$.
       \end{lemma}

\begin{proof} In our notations, the  statement of thesis of  the lemma is:
 \begin{equation}\label{N_R=K} 
\mathcal{N}_R(C_{\mathcal{K}})=\mathcal{K}.
\end{equation}
 By definition:
$$\mathcal{K}\subseteq \mathcal{N}_R(C_{\mathcal{K}}).$$
Let $v\in \mathcal{N}_R(C_{\mathcal{K}})$, then:
$$ B(Rv)\subset \cup_{w\in \mathcal{K}}B(Rw).$$
Let $\{w_n\}$ a sequence of points in $B(Rv)$ with limit $2Rv$. There exists a sequence of unit vectors $\{v_n\}$ such that  $w_n\in B(Rv_n)$. Up to a subsequence $v_n\to \overline{v}\in \mathcal{K}$ and $|w_n -2Rv_n| < R$. Then
$$|2Rv-R\overline{v}|=\lim_n |w_n-Rv_n|\leq R.$$
Thus
$ |2v-\overline{v}|^2 \leq 1;$ then:
$4+1-4\langle v ,\overline{v}\rangle \leq 1$ and 
$\langle v, \overline{v}\rangle \geq 1$. Therefore, since $v,\overline{v}$ are unit vectors,
 equality holds in Schwartz's inequality: 
$$\langle v, \overline{v}\rangle =|v||\overline{v}|.$$
Then, $v=\overline{v}\in \mathcal{K}$.
So  \eqref{N_R=K} is proved.
\end{proof}


\begin{theorem}\label{viceversaregular}Let $\mathcal{K}$ be a body in  $S^{d-1}$ 
and $K$ its  related cone in $\RR^d$. Let 
$\mathcal{N}=\Nor(C_{\mathcal{K}})\cap S^{d-1}$. Then
 \begin{enumerate}
 \item[a)] $\Tan(C_{\mathcal{K}})=-K^\star$;
   \item[b)]  $\mathcal{N}=co_{sph}(\mathcal{K})$;
   \item[c)] $ \Tan(C_{\mathcal{K}})\setminus \{o\}\subset int (C_{\mathcal{K}})$.
 
\end{enumerate} 
 \end{theorem}

\begin{proof} Let $w\in \mathcal{K}$ then 
$$C_{\mathcal{K}} \subset C_{\{w\}}$$
and 
\begin{equation}\label{tanintan}
 \Tan(C_{\mathcal{K}}) \subset \Tan( C_{\{w\}})=-\{w\}^\star.
\end{equation}
From \eqref{tanintan} it follows 
\begin{equation}\label{tanKstar}
 \Tan(C_{\mathcal{K}}) \subset -\bigcap_{w\in \mathcal{K}}\{w\}^\star=-K^\star.
\end{equation}
Let $\theta  \in K^\star$, then  $\forall v \in \mathcal{K}$, 
  $\langle \theta, v \rangle  \geq 0$.
  This implies that $\forall v \in \mathcal{K}  $ the ball $B(Rv)$ is $R$-supporting 
  the half line $\{-\la \theta, \la \geq 0\}$ at $o$  . Then
  \begin{equation}\label{K*inCk }
  \{-\la \theta, \la \geq 0\} \subset C_{\mathcal{K}}, \quad \forall  \theta  \in K^\star
  \end{equation}
    and therefore
  $$-\theta \in \Tan(C_{\mathcal{K}}).$$
 Then
 $$ -{K}^\star \subset \Tan(C_{\mathcal{K}}).$$
 From \eqref{tanKstar}   a) is proved. 
 

 Moreover from a) it follows that $\Tan(C_{\mathcal{K}})=-(co(K))^\star$; \eqref{nor=dualditan}  and 
 the bipolar theorem imply that
 $$N(C_\mathcal{K})=-(\Tan(C_{\mathcal{K}}))^\star=((co(K))^\star)^\star=co(K).$$ 
 b) follows.
 
 Let us prove now c). First let us prove that
  $$ \Tan(C_{\mathcal{K}})\subset C_{\mathcal{K}}.$$
 This inclusion  follows from  \eqref{N_R=K} and \eqref{TaninCNR}, with $A=C_{\mathcal{K}}$, $a=o$. If $o$ is an isolated point of $ C_{\mathcal{K}}$, then  $\Tan(C_{\mathcal{K}})=\{o\}$ and c) is trivial.
 In case $o$ is not an isolated point of $C_{\mathcal{K}}$, then let 
  $y\neq 0, y \in \Tan(C_{\mathcal{K}})$. By previous inclusion $y\in C_{\mathcal{K}}$. Let us prove that 
  $$\dist (y, (C_{\mathcal{K}})^c)= \dist (y,\cup_{v \in \mathcal{K}} (B(Rv,R))) > 0.$$
  By contradiction: if
  $\dist (y,\cup_{v \in \mathcal{K}} (B(Rv,R))) =0$, since $\mathcal{K}$ is compact, then there exists $u\in \mathcal{K}$, such that $y\in \pa B(Ru,R)$. Since $y\neq o$ and  
  $\pa B(Ru,R)$ strictly convex: $\langle y, u \rangle > 0$. Since by a): $y \in \Tan(C_{\mathcal{K}})=-K^*$, then $\langle y, u \rangle \leq  0$, contradiction.  
 \end{proof}
  
  
\begin{theorem}\label{viceversaregular2}Let $\mathcal{K}$ be a
body 
 in an hemisphere of $S^{d-1}$ and let $\mathcal{N}=\Nor (C_{\mathcal{K}},o)\cap S^{d-1}$. Then the following properties are equivalent: 
\item[a)] $\mathcal{K}$ is spherically convex ; 
\item[b)] $C_{\mathcal{K}}$ is a set of reach  greater or equal than $R$; 
\item[c)]   $C_{\mathcal{K}}=C_{\mathcal{N}}$.
 \end{theorem}

\begin{proof} By   \eqref{defA_R} and  \eqref{defunboundcone} :
 \begin{equation}\label{RK'=Ck}
   (R\mathcal{K})'_R=C_{\mathcal{K}}
 \end{equation}
 holds for every set $\mathcal{K}\subset S^{d-1}$.

Let us assume that a) holds.

If $\mathcal{K}$ is spherically convex   on a hemisphere of $S^{d-1}$,
it follows that $R\mathcal{K}$ is convex  on a hemisphere of a ball of radius $R$. Then
 for every $a,b \in R\mathcal{K}, |a-b|< 2R$ the set $R\mathcal{K}\cap \mathfrak{h}(a,b)$ is connected (see Definition \ref{htorto}).
Then by Proposition \ref{p1} 
$reach (R\mathcal{K})\geq R$. 

Thus $R\mathcal{K}$
 has $R$-hull and $R\mathcal{K}=co_R(R\mathcal{K})$.  By  ii) of Proposition  \ref{theorem4.4}
$$reach((R\mathcal{K})'_R) \geq R.$$ 
 This fact and equality \eqref{RK'=Ck} imply that  $C_{\mathcal{K}}$ 
is a set of reach greater or  equal than $R$; b) is proved.

  Let us assume  that b) holds.  
 
  From Theorem \ref{N=NRequality} any direction $v$ which lies 
in the normal cone at $o$   of $C_{\mathcal{K}} $ is $R$-supporting it at $o$; then $\forall v\in \mathcal{N}$:
 $$(B(Rv))^c \supset C_{\mathcal{K}}.$$
 Then 
  $C_{\mathcal{K}}\subset C_{\mathcal{N}} $. The opposite inclusion follows from \eqref{N=N_R}
  and c) is proved.

 Then by \eqref{N_R=K}  it follows that $\mathcal{K}=\mathcal{N}$. Therefore since by definition $\mathcal{N}$ is spherically convex then $\mathcal{K}$ is spherically convex too and a) follows.
 
 Let assume  that c) holds. 
 
 By   \eqref{N_R=K} the set  $\mathcal{K}$ is the set of the $R$-supporting unit vectors   of  $ C_{\mathcal{K}}$ at the origin $o$; similarly  $\mathcal{N}$ is the set of   the $R$-supporting unit vectors of 
  $ C_{\mathcal{N}}$ at $o$, then by c) $\mathcal{K}= \mathcal{N}$;  by b) of Theorem \ref{viceversaregular} a) follows.     \end{proof}
 
Next example shows that a $R$-supported body $A$, with $\mathcal{N}_R( A,p)$ spherically convex for every $P\in \pa A$, can be not a $R$-body, so it has reach 
less than $R$.  

\begin{example}Let  $P$ be a convex polygon  contained in a circle of radius  $R$, $A=\pa P$. Then  for every $p\in \pa A=A$,  $\mathcal{N}_R(A,p)$ is non empty and convex: if $p$ is inside a side of $P$,    $\mathcal{N}_R(\pa A,p)$ is a single   vector normal at  $p$ to the sides of $ P$; if $p$ is  a corner  of $P$, then $\mathcal{N}_R(\pa A,p)$  is spherically convex since it is  an arc  in  a semicircle.
 The body $A$ 
 is not a $R$-body so does not have reach greater or equal than $R$. 
\end{example}

 \begin{theorem}\label{regulararereach} Let $d\geq 2$, $A$ be a body in $\RR^d$. If  $reach(A)\geq R$ then 
 $A$ is a $R$-body and 
   for all $a\in \pa A$ the set of unit vectors $R$-supporting $A$ at $a$ is spherically convex.     
   \end{theorem}
\begin{proof} If $A$ has reach greater or equal than $R$, then by
Remark \ref{reach>=RimpliesRbodies} $A $ is a $R$-body.
 Theorem \ref  {RbodyfrontieraPaolo} and Lemma \ref{lemmainclusionnormalsets} prove the equality:
$$\mathcal{N} (A,a)=\mathcal{N}_R(A,a)$$
 for all $a\in \pa A$. Then, $\mathcal{N}_R(A,a)$ is convex  for all $a\in \pa A$.
\end{proof}
For the family of planar $R$-bodies the converse statement holds.
 \begin{theorem}\label{regulararereachd=2} 
Let $d=2$, let $A$ be a planar $R$-body. If, for all $a\in \pa A$, the set 
   $\mathcal{N}_R(A,a)$ is a spherically   convex set, then $A$ has reach greater or equal than $R$.
 \end{theorem}
 
 \begin{proof}Let us assume, by contradiction, that $reach(A) < R$.
By Proposition \ref{p1}, there exist $b_1,b_2\in A$, $|b_1-b_2| < 2R$, so that $A\cap\mathfrak{h} (b_1,b_2)$ is not connected. Then there exist $a_1,a_2 \in A\cap \mathfrak{h} (b_1,b_2)$, so that $A\cap \mathfrak{h} (a_1,a_2)= \{a_1, a_2\}$.

Let 
$$\mathfrak{h} (a_1,a_2) =cl(B(x_1))\cap cl(B(x_2))$$
and let
\begin{equation}\label{defH}
H(a_1,a_2)=B(x_1)\cup B(x_2).
\end{equation}
As $A$ is a $R$-body, by   \cite[Theorem 4.5 and lemma 4.1]{MLVRbodies}  
\begin{equation}\label{inclusionregolar}
 \mathfrak{h}(a_1,a_2)\setminus \{a_1,a_2\}\subset A^c
 \end{equation}
 implies
 $$H(a_1,a_2) \subset A^c.$$
From \eqref{defH} it follows 
$$B(x_1)\cup B(x_2) \subset A^c.$$
Since $a_i \in A\cap \pa B(x_i)$, $i=1,2$,  then $B(x_i)$  is   $R$-supporting  $A$ in $a_1$ and at $a_2$. Then 
$$\nu_i= \frac{x_i-a_1}{|x_i-a_1|}, \quad i=1,2$$
  are $R$-supporting vectors of $A$ at $a_1$.
As, by assumption, $\mathcal{N}_R(A,a_1)$ is spherically convex, then  it contains all unit vectors  connecting $\nu_1$ with $\nu_2$; then 
$$u=\frac{(a_2-a_1)}{|a_2-a_1|} \in  \mathcal{N}_R(A,a_1).$$
Since  $B(a_1+Ru) \ni a_2$, there is  a contradiction.
\end{proof}

 \section{Open  questions}
 Let us point out some open questions:
 \begin{itemize}
 \item[a)] Is Theorem \ref{regulararereachd=2} true for $d>2$? Let us notice that it is true for a $R$-cone, Theorem \ref{viceversaregular2}.
 \item[b)] Let $E\subset \R^d$, $d> 2$,  be a connected body, contained in an open ball of radius $R$, then is $co_R(E)$  connected? For $d=2$ the statement is true, see \cite[Theorem 4.8]{MLVRbodies}.
 \item[c)] If $E$ is the set $V$ of the vertices of a simplex in $\R^d$, the boundary of the $R$-hulloid $co_R(V)$  has properties which follow from Theorem  \ref{theorem1suportingball}. Is it  possible to describe completely the shape of $co_R(V)$? 
 
 In two dimensions this description is made in  \cite[Theorem 4.2]{MLVRbodies}. 
 In a forthcoming paper:  
  M. Longinetti, S. Naldi and  A. Venturi, 
\emph{$R$-hulloid of the vertices of a tetrahedron},  a complete   
 characterization of the shape of $co_R(V)$ will be  given.  

 \end{itemize}

\section*{Funding}
This work has been partially supported by INDAM-GNAMPA(2023).

\end{document}